\documentclass[12pt,psamsfonts]{amsart}
\usepackage{amssymb,eucal,epsfig,amsmath}

\usepackage{pgf,tikz}
\usetikzlibrary{arrows}

\def\Z{\mathbb Z}

\def\R{\mathbb R}

\theoremstyle{plain}
\newtheorem{theorem}{Theorem}

\newtheorem*{KSWT}{The van Kampen--Shapiro--Wu Theorem}

\newtheorem*{WT}{The Wu--Tutte Theorem}

\newtheorem{lemma}{Lemma}

\theoremstyle{definition}
\newtheorem{definition}{Definition}

\theoremstyle{remark}
\newtheorem{remark}{Remark}

\title{Bad drawings of small complete graphs}
\author{Grant Cairns}
\author{Emily Groves}
\author{Yuri Nikolayevsky}

\address{Department of Mathematics, La Trobe University, Melbourne, Australia 3086}
\email{G.Cairns@latrobe.edu.au}
\email{18908496@students.latrobe.edu.au}
\email{Y.Nikolayevsky@latrobe.edu.au}

\begin{document}

\maketitle

\begin{abstract}
  We show that for  $K_5$ (resp.~ $K_{3,3}$) there is a drawing with $i$ independent crossings, and no pair of independent edges cross more than once,  provided $i$ is odd with $1\le i\le 15$ (resp.~ $1\le i\le 17$). Conversely, using the deleted product cohomology, we show that for $K_5$ and $K_{3,3}$, if $A$ is any set of pairs of independent edges, and $A$ has odd cardinality, then there is a drawing in the plane for which each element in $A$ cross an odd number of times, while  each pair of independent edges not in $A$ cross an even number of times. 
For $K_6$ we show that there is a drawing with $i$ independent crossings, and no pair of independent edges cross more than once,  if and only if  $3\le i\le 40$. 
\end{abstract}

\section{Introduction}

We consider planar drawings of  finite simple graphs in which vertices are represented as points, the edges are smooth arcs, joining distinct vertices, that do not self-intersect or pass through any vertex, and when distinct edges meet they only do so at common vertex endpoints, or at transverse crossings and in the latter case only have finitely many such crossings. 
Recall that two edges are said to be \emph{independent} if they are distinct and not adjacent, and a drawing is \emph{good}
if no pair of adjacent edges cross one another, and each pair of independent edges cross at most once  \cite{guy}. A crossing of two independent edges is called an \emph{independent crossing}.

\begin{definition} We say that a graph drawing is \emph{bad} if it is not good, but that it is \emph{tolerable} 
if no pair of independent edges cross more than once.
\end{definition}


Obviously, all good drawings are tolerable and as we will see, there are more tolerable drawings than good ones. For example, it is easy to see that in any good drawing of $K_4$, there is at most one crossing; it follows that  good drawings of $K_n$ have at most $\binom{n}{4}$ crossings \cite{Aetc} and this upper bound is attained for a straight-line drawing with the vertices at the vertices of a regular $n$-gon. 
But there are tolerable drawings of $K_4$ in which all 3 pairs of  independent edges cross; see Figure \ref{F:k4}.

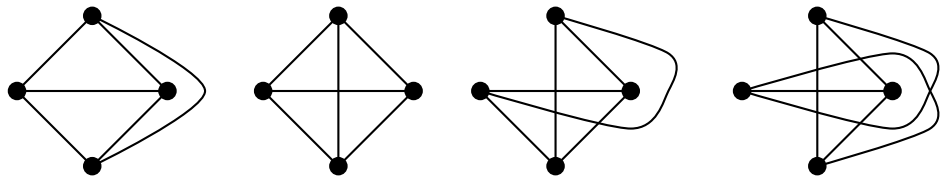
\begin{figure}[h!]
\begin{center}
\definecolor{blue}{rgb}{0,0,1}
\begin{tikzpicture}[scale=1];
\draw [thick,-]  (1.,0) -- (-1.,0)  ;
\draw [thick,-]  (1.,0) -- (0,1)  ;
\draw [thick,-]  (1.,0) -- (0,-1)  ;
\draw [thick,-]  (-1.,0) -- (0,-1)  ;
\draw [thick] (0.,1) -- (-1.,0)  ;
\draw [thick] plot [smooth] coordinates { (0,-1) (1.5,0)   (0,1)};
\fill  (1.,0) circle (3.5pt);
\fill  (0,-1) circle (3.5pt);
\fill  (0,1) circle (3.5pt);
\fill  (-1.,0) circle (3.5pt);
\end{tikzpicture}
\hskip.5cm
\begin{tikzpicture}[scale=1];
\draw [thick,-]  (1.,0) -- (-1.,0)  ;
\draw [thick,-]  (1.,0) -- (0,1)  ;
\draw [thick,-]  (1.,0) -- (0,-1)  ;
\draw [thick,-]  (0,1) -- (0,-1)  ;
\draw [thick,-]  (-1.,0) -- (0,-1)  ;
\draw [thick] (0.,1) -- (-1.,0)  ;
\fill  (1.,0) circle (3.5pt);
\fill  (0,-1) circle (3.5pt);
\fill  (0,1) circle (3.5pt);
\fill  (-1.,0) circle (3.5pt);
\end{tikzpicture}
\hskip.5cm
\begin{tikzpicture}[scale=1];
\draw [thick,-]  (1.,0) -- (-1.,0)  ;
\draw [thick,-]  (1.,0) -- (0,1)  ;
\draw [thick,-]  (1.,0) -- (0,-1)  ;
\draw [thick,-]  (0,1) -- (0,-1)  ;
\draw [thick,-]  (-1.,0) -- (0,-1)  ;
\draw [thick] plot [smooth] coordinates { (-1.,0) (1,-.5)  (1.5,0) (1.5,0.5)  (0,1)};
\fill  (1.,0) circle (3.5pt);
\fill  (0,-1) circle (3.5pt);
\fill  (0,1) circle (3.5pt);
\fill  (-1.,0) circle (3.5pt);
\end{tikzpicture}
\hskip.5cm
\begin{tikzpicture}[scale=1];
\draw [thick,-]  (1.,0) -- (-1.,0)  ;
\draw [thick,-]  (1.,0) -- (0,1)  ;
\draw [thick,-]  (1.,0) -- (0,-1)  ;
\draw [thick,-]  (0,1) -- (0,-1)  ;
\draw [thick] plot [smooth, tension=.5] coordinates { (-1.,0) (1,.5)  (1.5,0) (1.5,-0.5)  (0,-1)};
\draw [thick] plot [smooth, tension=.5] coordinates { (-1.,0) (1,-.5)  (1.5,0) (1.5,0.5)  (0,1)};
\fill  (1.,0) circle (3.5pt);
\fill  (0,-1) circle (3.5pt);
\fill  (0,1) circle (3.5pt);
\fill  (-1.,0) circle (3.5pt);
\end{tikzpicture}
\end{center}
\caption{Tolerable drawings of $K_4$ with zero to three independent crossings respectively}
\label{F:k4}
\end{figure}

In this paper we study bad drawings and tolerable drawings of small complete graphs, and small complete bipartite graphs. We present two kinds of results. The first kind concerns the existence of tolerable drawings having a certain number of independent crossings.

\begin{theorem}\label{K5K33K6Thm}\ 
\begin{enumerate}
\item For each odd integer $i$ with $1\le i\le 15$, there is a tolerable drawing of $K_5$ with $i$ independent crossings.
\item For each odd integer $i$ with $1\le i\le 17$, there is a tolerable drawing of $K_{3,3}$ with $i$ independent crossings.
\item For each integer $i$ with $3\le i\le 40$, there is a tolerable drawing of $K_6$ with $i$ independent crossings.
\end{enumerate}
 \end{theorem}

The existence of drawings described in the above theorem  is presented explicitly in Section~\ref{K5K33K6}. Conversely, we will see below that for each of these graphs, there are no tolerable drawings having a number of independent crossing other than those indicated in Theorem~\ref{K5K33K6Thm}. In order to explain this in greater detail, we require some terminology.

\begin{definition} For a graph $G$, let $PG$ denote the set of pairs of independent edges of $G$. We will say that a subset $A$ of $PG$ is \emph{2-realisable} if there is a drawing of $G$ in the plane for which each element in $A$ cross an odd number of times, while  each element of $PG\backslash A$ cross an even number of times. Further, we say that such a drawing \emph{2-realises} $A$. 
\end{definition}

To avoid any confusion, let us emphasise that in the above definition we impose no restrictions on the numbers of crossings of pairs of adjacent edges. For a given graph $G$ and given subset $A$ of $PG$, it is natural to ask whether $A$ is 2-realisable, and if so, is there a tolerable, or even good, drawing that 2-realises $A$. For example, for $K_4$, there are 6 edges and the set $PK_4$ of independent pairs has three elements. So there are $2^3=8$ possible subsets of $PK_4$. However, exploiting the $S_4$ symmetry, one easily sees that up to a relabelling of the vertices, there are just 4 essentially different subsets of $PK_4$, having 0, 1, 2, 3 elements respectively. As shown in Figure \ref{F:k4}, these subsets are all 2-realisable. In fact, the subsets having 0 or 1 element have a good drawing, while the subsets having 2 or 3 elements have tolerable drawings, but no good ones. 

If $G$ is  $K_5,K_{3,3}$ or $K_6$, and $A\subset PG$ is 2-realisable, we will see that the cardinality of $A$ satisfies the corresponding inequality in Theorem~\ref{K5K33K6Thm}. For $K_5$ and $K_{3,3}$, this result is immediate from Kleitman's Theorem (see Section~\ref{Ks}). We present the stronger  statement (see Theorem~\ref{T:sub}):

\begin{theorem}\label{K5K33con}\ 
If $G$ is  $K_5$ or $K_{3,3}$ and $A\subset PG$, then $A$ is 2-realisable if and only if its cardinality satisfies the corresponding condition of  Theorem~\ref{K5K33K6Thm}(a) or (b).
 \end{theorem}

A similar result does not hold for $K_6$; for example, as we explain in Section~\ref{higherKs}, there are 3-element subsets of $PK_6$ that are 2-realisable and there are 3-element subsets of $PK_6$ that are not 2-realisable. However, one does have:

\begin{theorem}\label{K6con}\ 
If  $A\subset PK_6$ is 2-realisable, then the cardinality of $A$ satisfies the condition in Theorem~\ref{K5K33K6Thm}(c).
 \end{theorem}

Our proofs of Theorems \ref{K5K33con} and \ref{K6con} use the deleted product cohomology machinery. We recall this briefly in Section \ref{Wu}. 

Notice for $G=K_5$ and $G=K_{3,3}$, the above results do not claim that tolerable drawings exist for every 2-realisable set  $A \subset PG$. Indeed, we strongly suspect that this is not the case; see Section~\ref{intolerable}.
There do exist graphs $G$ and 2-realisable subsets $A\subset PG$ such that $A$ is not 2-realised by any tolerable drawing. To see this, recall that 
for a given graph $G$, the \emph{crossing number} of $G$ is the minimum number of edge intersections in a plane drawing of the graph, where each intersection is counted separately, but if one counts the number of pairs of edges that intersect an odd number of times, you get the \emph{odd crossing number}. In \cite{PSS2008}, Pelsmajer,  Schaefer  and \v{S}tefankovi\v{c} gave a recipe for producing an example of a graph $G$ for which the odd crossing number is strictly less than the crossing number. For a drawing that realises the odd crossing number, let $A$ be the set of pairs of independent edges that meet an odd number of times. So by definition, the given drawing is a 2-realisation of $A$. But there can be no tolerable drawing that 2-realises $A$ since otherwise the crossing number would equal the odd crossing number. In Section~\ref{intolerable} we give further  examples of  graphs $G$ and subsets $A\subset PG$ which are 2-realisable but  not 2-realised by any tolerable drawing.

In order to give the reader further familiarity with the concepts, we begin in the following section  with a complete account of the graph $K_{2,3}$, where the situation is sufficiently simply that calculations can be readily done by hand. We show that every subset of $PK_{2,3}$ is 2-realised by a tolerable drawing. 

\section{The complete bipartite graph $K_{2,3}$}\label{Section:K23}

The graph $K_{2,3}$ has  6  edges. For each edge, there are two independent edges, so this gives 6 pairs of independent edges. Thus $PK_{2,3}$ has $2^6=64$ subsets. For $G=K_{2,3}$, the symmetry group $S_3\times \Z_2$ has order 12. It acts naturally on $PK_{2,3}$ and one can compute the number of orbits using the ``lemma that is not Burnside's'' \cite{Neu}. It is clear that by taking complements, the number of orbits of subsets of $PK_{2,3}$ of cardinality $i$ is the same as the number of orbits of subsets of cardinality $6-i$. Furthermore, the group $S_3\times \Z_2$  clearly acts transitively on the set of subsets having just one element. So it suffices to consider the action of $S_3\times \Z_2$   on the set of subsets of $PK_{2,3}$ of cardinality 2, and on the set of subsets of cardinality 3. The number of fixed points for the action is given in Table \ref{table:K23}; the vertices in the two parts of $K_{2,3}$ are labelled 1, 2, 3 and 4, 5 respectively. In the first column we have typical elements of conjugacy classes; in the second, the number of elements in the conjugacy class. From this we have that for subsets of cardinality 2 there are $36/12=3$ orbits; representatives for these orbits are as follows:
\[
\{(14)(25), (14)(35)\}, \{(14)(25), (24)(35)\}, \{(14)(25)\ (15)(24)\}.
\]
Similarly, for subsets of cardinality 3 there are also $36/12=3$ orbits, and representatives for these orbits are as follows:
\[
\{(14)(25), (14)(35), (24)(35)\}, 
\{(14)(25), (15)(24), (24)(35)\}, 
\{(14)(25), (14)(35), (15)(34)\}. 
\]
Hence, up to relabelling, there are 13 essentially distinct subsets $A$ of  $PK_{2,3}$; there are 1, 1, 3, 3, 3, 1, 1 such subsets having 0, 1, 2, 3, 4, 5, 6 elements respectively. Each of these subsets is 2-realised by a tolerable drawing, as shown in Figure~\ref{F:k23}; note that the cardinalities of $A$ in this Figure are not in increasing order. The first 5 of these are good drawings. The other subsets can't be 2-realised by good drawings. Indeed,  there are only 6 good drawings of $K_{2,3}$ up to isomorphism \cite{Harbo}, and two of these (with 3 crossings) correspond to the same subset of $PK_{2,3}$.

\begin{remark}
Notice that in all the drawings in Figure~\ref{F:k23}, one can draw a vertical edge between the two left-most vertices without crossing any other edge. The resulting graph is $K_{1,1,3}$ and the new edge is not part of any independent pair. Thus Figure~\ref{F:k23} shows that every subset of $PK_{1,1,3}$ can be 2-realised by a tolerable drawing.
\end{remark}

\begin{table}[h]
\begin{center}
\begin{tabular}{c|c|c|c}
  \hline
  typical element & $\#$ elements& $\#$ fixed subsets of card. 2& $\#$ fixed subsets of card. 3 \\
  \hline
 id & 1 & $\binom{6}{2}=15$ & $\binom{6}{3}=20$  \\
    $(12)$ & 3 & 3 & 2  \\
    $(123)$ & 2 & 0 & 2  \\
    $(45)$ & 1 & 3 & 0  \\
    $(12)(45)$ & 3 & 3 & 2  \\
    $(123)(45)$ & 2 & 0 & 0  \\
  \hline
     & 12 & 36 & 36  \\
  \hline
\end{tabular}
\end{center}
\caption{\ }
\label{table:K23}
\end{table}


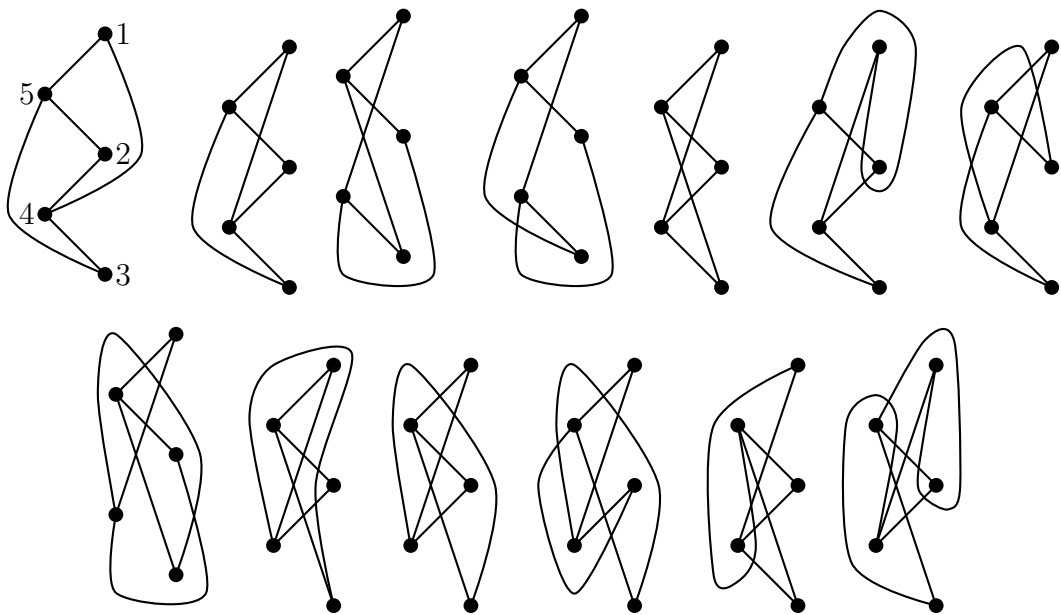
\begin{figure}[h!]
\begin{center}
\definecolor{blue}{rgb}{0,0,1}
\begin{tikzpicture}[scale=.8];
\newcommand\va{(1,2)};
\newcommand\vb{(1,0)};
\newcommand\vc{(1,-2)};
\newcommand\vd{(0,-1)};
\newcommand\ve{(0,1)};
\draw [thick,-]  \va -- \ve ;
\draw [thick,-]  \vb -- \ve ;
\draw [thick,-]  \vb -- \vd ;
\draw [thick,-]  \vc -- \vd ;
\draw [thick] plot [smooth, tension=.5] coordinates { \va  (1.6,0) \vd };
\draw [thick] plot [smooth, tension=.5] coordinates { \vc  (-.6,-1) \ve };
\fill  \va circle (3.5pt);
\fill  \vb circle (3.5pt);
\fill  \vc circle (3.5pt);
\fill  \vd circle (3.5pt);
\fill  \ve circle (3.5pt);
\draw (1.3,2) node {$1$};
\draw (1.3,0) node {$2$};
\draw (1.3,-2) node {$3$};
\draw (-.3,1) node {$5$};
\draw (-.3,-1) node {$4$};
\end{tikzpicture}
\hskip.3cm
\begin{tikzpicture}[scale=.8];
\newcommand\va{(1,2)};
\newcommand\vb{(1,0)};
\newcommand\vc{(1,-2)};
\newcommand\vd{(0,-1)};
\newcommand\ve{(0,1)};
\draw [thick,-]  \va -- \ve ;
\draw [thick,-]  \vb -- \ve ;
\draw [thick,-]  \vb -- \vd ;
\draw [thick,-]  \vc -- \vd ;
\draw [thick,-]  \va -- \vd ;
\draw [thick] plot [smooth, tension=.5] coordinates { \vc  (-.6,-1) \ve };
\fill  \va circle (3.5pt);
\fill  \vb circle (3.5pt);
\fill  \vc circle (3.5pt);
\fill  \vd circle (3.5pt);
\fill  \ve circle (3.5pt);
\end{tikzpicture}
\hskip.3cm
\begin{tikzpicture}[scale=.8];
\newcommand\va{(1,2)};
\newcommand\vb{(1,0)};
\newcommand\vc{(1,-2)};
\newcommand\vd{(0,-1)};
\newcommand\ve{(0,1)};
\draw [thick,-]  \va -- \ve ;
\draw [thick,-]  \vb -- \ve ;
\draw [thick,-]  \vc -- \ve ;
\draw [thick,-]  \vc -- \vd ;
\draw [thick,-]  \va -- \vd ;
\draw [thick] plot [smooth, tension=.5] coordinates { \vb  (1.5,-2.3) (0,-2.3) \vd };
\fill  \va circle (3.5pt);
\fill  \vb circle (3.5pt);
\fill  \vc circle (3.5pt);
\fill  \vd circle (3.5pt);
\fill  \ve circle (3.5pt);
\end{tikzpicture}
\hskip.3cm
\begin{tikzpicture}[scale=.8];
\newcommand\va{(1,2)};
\newcommand\vb{(1,0)};
\newcommand\vc{(1,-2)};
\newcommand\vd{(0,-1)};
\newcommand\ve{(0,1)};
\draw [thick,-]  \va -- \ve ;
\draw [thick,-]  \vb -- \ve ;
\draw [thick,-]  \vc -- \vd ;
\draw [thick,-]  \va -- \vd ;
\draw [thick] plot [smooth, tension=.5] coordinates { \vc  (-.6,-1) \ve };
\draw [thick] plot [smooth, tension=.5] coordinates { \vb  (1.5,-2.3) (0,-2.3) \vd };
\fill  \va circle (3.5pt);
\fill  \vb circle (3.5pt);
\fill  \vc circle (3.5pt);
\fill  \vd circle (3.5pt);
\fill  \ve circle (3.5pt);
\end{tikzpicture}
\hskip.3cm
\begin{tikzpicture}[scale=.8];
\newcommand\va{(1,2)};
\newcommand\vb{(1,0)};
\newcommand\vc{(1,-2)};
\newcommand\vd{(0,-1)};
\newcommand\ve{(0,1)};
\draw [thick,-]  \va -- \ve ;
\draw [thick,-]  \vb -- \ve ;
\draw [thick,-]  \vc -- \vd ;
\draw [thick,-]  \va -- \vd ;
\draw [thick,-]  \vc -- \ve ;
\draw [thick,-]  \vb -- \vd ;
\fill  \va circle (3.5pt);
\fill  \vb circle (3.5pt);
\fill  \vc circle (3.5pt);
\fill  \vd circle (3.5pt);
\fill  \ve circle (3.5pt);
\end{tikzpicture}
\hskip.3cm
\begin{tikzpicture}[scale=.8];
\newcommand\va{(1,2)};
\newcommand\vb{(1,0)};
\newcommand\vc{(1,-2)};
\newcommand\vd{(0,-1)};
\newcommand\ve{(0,1)};
\draw [thick,-]  \vb -- \ve ;
\draw [thick,-]  \vb -- \vd ;
\draw [thick,-]  \vc -- \vd ;
\draw [thick,-]  \va -- \vd ;
\draw [thick] plot [smooth, tension=.5] coordinates { \va  (.7,0) (1,-.4) (1.3,0)(1.6,2) (1,2.6) (.4,2)\ve };
\draw [thick] plot [smooth, tension=.5] coordinates { \vc  (-.8,-1) \ve };
\fill  \va circle (3.5pt);
\fill  \vb circle (3.5pt);
\fill  \vc circle (3.5pt);
\fill  \vd circle (3.5pt);
\fill  \ve circle (3.5pt);
\end{tikzpicture}
\hskip.3cm
\begin{tikzpicture}[scale=.8];
\newcommand\va{(1,2)};
\newcommand\vb{(1,0)};
\newcommand\vc{(1,-2)};
\newcommand\vd{(0,-1)};
\newcommand\ve{(0,1)};
\draw [thick,-]  \va -- \ve ;
\draw [thick,-]  \vb -- \ve ;
\draw [thick,-]  \vc -- \vd ;
\draw [thick,-]  \va -- \vd ;
\draw [thick] plot [smooth, tension=.5] coordinates { \vb  (.5,2) (-.5,1)   \vd };
\draw [thick] plot [smooth, tension=.5] coordinates { \vc   (-.5,-1)\ve };
\fill  \va circle (3.5pt);
\fill  \vb circle (3.5pt);
\fill  \vc circle (3.5pt);
\fill  \vd circle (3.5pt);
\fill  \ve circle (3.5pt);
\end{tikzpicture}

\vskip.3cm
\begin{tikzpicture}[scale=.8];
\newcommand\va{(1,2)};
\newcommand\vb{(1,0)};
\newcommand\vc{(1,-2)};
\newcommand\vd{(0,-1)};
\newcommand\ve{(0,1)};
\draw [thick,-]  \va -- \ve ;
\draw [thick,-]  \vb -- \ve ;
\draw [thick,-]  \va -- \vd ;
\draw [thick,-]  \vc -- \ve ;
\draw [thick] plot [smooth, tension=.5] coordinates { \vc  (1.4,0) (0,2)  (-.3,1)   \vd };
\draw [thick] plot [smooth, tension=.5] coordinates { \vb  (1.5,-2.3) (0,-2.3) \vd };
\fill  \va circle (3.5pt);
\fill  \vb circle (3.5pt);
\fill  \vc circle (3.5pt);
\fill  \vd circle (3.5pt);
\fill  \ve circle (3.5pt);
\end{tikzpicture}
\hskip.3cm
\begin{tikzpicture}[scale=.8];
\newcommand\va{(1,2)};
\newcommand\vb{(1,0)};
\newcommand\vc{(1,-2)};
\newcommand\vd{(0,-1)};
\newcommand\ve{(0,1)};
\draw [thick,-]  \va -- \ve ;
\draw [thick,-]  \va -- \vd ;
\draw [thick,-]  \vb -- \vd ;
\draw [thick,-]  \vb -- \ve ;
\draw [thick,-]  \vc -- \ve ;
\draw [thick] plot [smooth, tension=.5] coordinates { \vc  (.7,0) (1.3,2.2)(0,2)(-.4,1.1)   \vd };
\fill  \va circle (3.5pt);
\fill  \vb circle (3.5pt);
\fill  \vc circle (3.5pt);
\fill  \vd circle (3.5pt);
\fill  \ve circle (3.5pt);
\end{tikzpicture}
\hskip.3cm
\begin{tikzpicture}[scale=.8];
\newcommand\va{(1,2)};
\newcommand\vb{(1,0)};
\newcommand\vc{(1,-2)};
\newcommand\vd{(0,-1)};
\newcommand\ve{(0,1)};
\draw [thick,-]  \va -- \ve ;
\draw [thick,-]  \vb -- \ve ;
\draw [thick,-]  \va -- \vd ;
\draw [thick,-]  \vc -- \ve ;
\draw [thick,-]  \vb -- \vd ;
\draw [thick] plot [smooth, tension=.5] coordinates { \vc  (1.4,0) (0,2)  (-.3,1)   \vd };
\fill  \va circle (3.5pt);
\fill  \vb circle (3.5pt);
\fill  \vc circle (3.5pt);
\fill  \vd circle (3.5pt);
\fill  \ve circle (3.5pt);
\end{tikzpicture}
\hskip.3cm
\begin{tikzpicture}[scale=.8];
\newcommand\va{(1,2)};
\newcommand\vb{(1,0)};
\newcommand\vc{(1,-2)};
\newcommand\vd{(0,-1)};
\newcommand\ve{(0,1)};
\draw [thick,-]  \va -- \ve ;
\draw [thick,-]  \va -- \vd ;
\draw [thick,-]  \vc -- \ve ;
\draw [thick,-]  \vb -- \vd ;
\draw [thick] plot [smooth, tension=.5] coordinates { \vc  (1.4,0) (0,2)  (-.3,1)   \vd };
\draw [thick] plot [smooth, tension=.5] coordinates { \vb  (0,-1.8) (-.6,0)    \ve };
\fill  \va circle (3.5pt);
\fill  \vb circle (3.5pt);
\fill  \vc circle (3.5pt);
\fill  \vd circle (3.5pt);
\fill  \ve circle (3.5pt);
\end{tikzpicture}
\hskip.3cm
\begin{tikzpicture}[scale=.8];
\newcommand\va{(1,2)};
\newcommand\vb{(1,0)};
\newcommand\vc{(1,-2)};
\newcommand\vd{(0,-1)};
\newcommand\ve{(0,1)};
\draw [thick,-]  \va -- \vd ;
\draw [thick,-]  \vc -- \ve ;
\draw [thick,-]  \vb -- \ve ;
\draw [thick,-]  \vb -- \vd ;
\draw [thick,-]  \vc -- \vd ;
\draw [thick] plot [smooth, tension=.5] coordinates { \va  (-.4,1) (-.4,-1.5) (0,-1.6)  (.3,-1) \ve };
\fill  \va circle (3.5pt);
\fill  \vb circle (3.5pt);
\fill  \vc circle (3.5pt);
\fill  \vd circle (3.5pt);
\fill  \ve circle (3.5pt);
\end{tikzpicture}
\hskip.3cm
\begin{tikzpicture}[scale=.8];
\newcommand\va{(1,2)};
\newcommand\vb{(1,0)};
\newcommand\vc{(1,-2)};
\newcommand\vd{(0,-1)};
\newcommand\ve{(0,1)};
\draw [thick,-]  \va -- \vd ;
\draw [thick,-]  \vc -- \ve ;
\draw [thick,-]  \vb -- \ve ;
\draw [thick,-]  \vb -- \vd ;
\draw [thick] plot [smooth, tension=.5] coordinates { \va  (.7,.2) (.8,-.2) (1.2,-.4) (1.4,0)  (1.3,2.4)(.8,2.4)\ve };
\draw [thick] plot [smooth, tension=.5] coordinates { \vc  (-.4,-1.3) (.-.5,1) (0,1.5) (.35,1)  \vd };
\fill  \va circle (3.5pt);
\fill  \vb circle (3.5pt);
\fill  \vc circle (3.5pt);
\fill  \vd circle (3.5pt);
\fill  \ve circle (3.5pt);
\end{tikzpicture}
\end{center}
\caption{Tolerable drawings of $K_{2,3}$ }
\label{F:k23}
\end{figure}


\section{Intolerably bad examples}\label{intolerable}

In this section, we exhibit a graph $G$ and a subset of $PG$ that is 2-realisable but cannot be 2-realised by a tolerable drawing of $G$. Our graph $G$ will be a disjoint union of $N$ edges. Since $G$ has no adjacent edges, we have full control on its crossing pattern, but having constructed such an example we will be able to construct a connected example, as explained in Remark~\ref{r:connect} below.

\begin{lemma} \label{l:all2}
  If $G$ is a disjoint union of edges, then any subset $A$ of the set $PG$ of pairs of its independent edges is $2$-realisable.
\end{lemma}
\begin{proof}
First draw the edges of $G$ as disjoint line segments on the plane. Next, for every pair of edges $\{e,f\} \in A$, join an interior point of $e$ to one of the endpoints $v$ of $f$ by a simple curve $\gamma$ whose interior does not meet any edge or vertex, and then replace $\gamma$ by a curve passing along the boundary of a thin strip centred on $\gamma$ and a small semicircle centred at $v$, as in Figure \ref{F:epsilonnbd}. Complete this construction successively for each of the required curves $\gamma$, ensuring that at each step the curve only meets previously drawn curves at transverse crossings and that there are at most a finite number of such crossings, and furthermore, no curves starting on a common edge meet at all. 
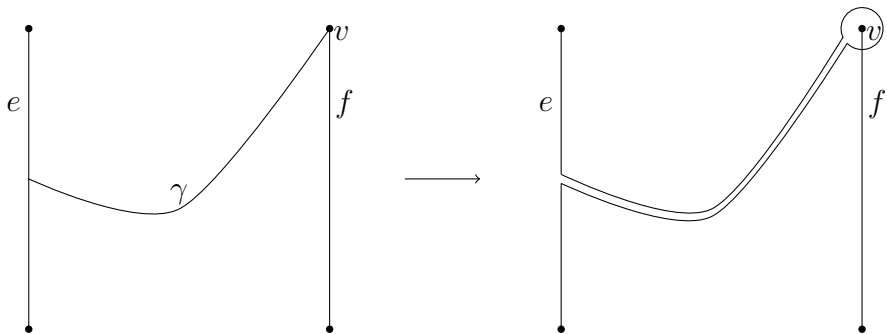
\begin{figure}[h!]
\begin{center}
\begin{tikzpicture}[scale=2];
\draw (0,0) -- (0,2);
\draw (2,0) -- (2,2);
\draw (-.1,1.5) node {$e$};
\draw (2.1,1.5) node {$f$};
\draw (2.08,1.97) node {$v$};
\draw (1,.9) node {$\gamma$};
\draw  plot [smooth, tension=.5] coordinates {  (0,1) (1,.8) (2,2) };
\fill  (2,0) circle (.7pt);
\fill  (2,2) circle (.7pt);
\fill  (0,0) circle (.7pt);
\fill  (0,2) circle (.7pt);
\draw[->] (2.5,1) -- (3,1);
\end{tikzpicture}
\hskip.5cm
\begin{tikzpicture}[scale=2];
\draw (0,0) -- (0,.97);
\draw (0,1.03) -- (0,2);
\draw (2,0) -- (2,2);
\draw  plot [smooth, tension=.5] coordinates {  (0,1.03) (1,.8) (1.875,1.945) };
\draw  plot [smooth, tension=.5] coordinates { (0,.97) (1,.75) (1.9,1.9) };
\draw [domain=-136:205] plot ({2+.14*cos(\x)}, {2+.14*sin(\x)});
\draw (-.1,1.5) node {$e$};
\draw (2.1,1.5) node {$f$};
\draw (2.08,1.97) node {$v$};
\fill  (2,0) circle (.7pt);
\fill  (2,2) circle (.7pt);
\fill  (0,0) circle (.7pt);
\fill  (0,2) circle (.7pt);
\end{tikzpicture}
\end{center}
\caption{Construction of $2$-realisable drawings}
\label{F:epsilonnbd}
\end{figure}
\end{proof}

\begin{remark} \label{r:connect}
Suppose we have a graph $G$ which is a disjoint union of $N$ edges and a subset $A\subset PG$ that cannot be 2-realised by a tolerable drawing. Take a cycle $G'$ of length $2N$ with $G$ being its subgraph of even labelled edges (in some cyclic order). Consider an arbitrary $2$-realisation of $A$ and add to it odd labelled edges of $G'$ arbitrarily to get a drawing of $G'$. Now define the subset $A'$ of pairs of independent edges of $G'$ from the drawing (a pair belongs to $A'$ if the edges cross an odd number of times). Then $A'$ is $2$-realisable by design, but the drawing is not tolerable, because if it were, then the drawing of $G$ corresponding to $A$ would have also been tolerable.
\end{remark}

For $n, m \in \mathbb{N}$, denote $I=\{1, 2, \dots, n\}$ and let $S$ be a set,  of cardinality $m$, of subsets of $I$. We define $G$ to be the union of $N = n + m$ pairwise disjoint edges which we label $e_i, \, i \in I, \; e_s, \, s \in S$. We then define $A$ to be the set of pairs $\{i,s\}$, where $i \in s$. In a tolerable drawing, the first $n$ edges must be pairwise disjoint, the last $m$ edges must be pairwise disjoint, and then one of the first $n$ edges $e_i$ crosses one of the last $m$ edges $e_s$ if and only if $i \in s$.

\begin{lemma} \label{l:intol}
  In the above notation, for $n=5$ and $S=\{I\}\cup \{(i,j) :  1\le i<j\le n\}$, there is no tolerable drawing that 2-realises $A$.
\end{lemma} 

\begin{proof}
Suppose there exists a tolerable drawing that 2-realises $A$.
The graph $G$ here is
  the disjoint union of 16 edges.  Up to isotopy, the drawing of the edges $e_1, \dots, e_5, e_I$ looks like the one in Figure~\ref{F:e1nI}, where we relabel the edges $e_1, e_2, \dots, e_5$ if necessary in such a way that the edge $e_I$ crosses them in the increasing (or decreasing) order of labels.

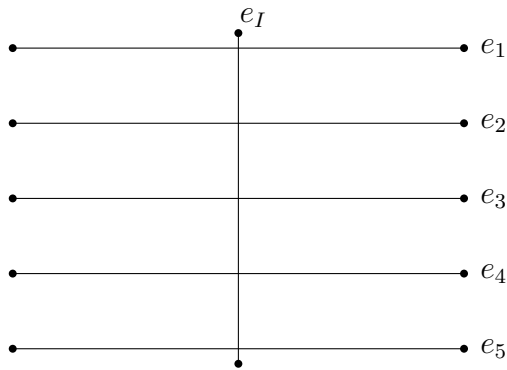
\begin{figure}[h!]
\begin{center}
\begin{tikzpicture}
\foreach \y in {3,2,1,0,-1} \draw (-3,\y) -- (3,\y)  ;
\foreach \y in {-1.2,3.2} \foreach \x in {0}
\fill  ({\x},{\y}) circle (1.5pt);
\foreach \y in {3,2,1,0,-1} \foreach \x in {-3,3}
\fill  ({\x},{\y}) circle (1.5pt);
\draw (0,-1.2) -- (0,3.2);
\draw(0.2,3.4) node {$e_I$};
\draw(3.4,3) node {$e_1$}; 
\draw(3.4,2) node {$e_2$}; 
\draw(3.4,1) node {$e_{3}$};
\draw(3.4,0) node {$e_{4}$};
 \draw(3.4,-1) node {$e_5$};
\end{tikzpicture}
\end{center}
\caption{Drawing of the edges $e_1, \dots, e_5, e_I$}
\label{F:e1nI}
\end{figure}

We can ignore the edges $e_{k,k+1}, k=1, 2, 3, 4$, as they can be inserted at any stage. This leaves us with the six edges, $e_{1,3}, e_{1,4}, e_{1,5}, e_{2,4}, e_{2,5}$ and $e_{3,5}$. Now for every $s=\{i,j\} \in S, \; i \ne j$, the edge $e_s$ crosses the edges $e_i$ and $e_j$, once each, and does not cross any other edges of $G$. Referring to Figure~\ref{F:e1nI} we say that the edge $e_s$ is \emph{left} (respectively \emph{right}) if the crossings of both $e_i$ and $e_j$ with $e_s$ occur to the left (respectively to the right) of their crossings with $e_I$. Otherwise call the edge $e_s$ \emph{mixed}.
Consider $e_{1,5}$. We need only consider the two possibilities that $e_{1,5}$ is either mixed or left since we can use the symmetry group $\Z_2 \times \Z_2$ given by the reflection about the line containing $e_I$ and the reflection (followed by relabelling the $e_i$'s in the reverse order) about the line containing $e_3$. If $e_{1,5}$ is mixed, then $e_{2,4}$ must be right (up to symmetry), and then $e_{1,3}$ is left, and then $e_{2,5}$ is right, and then $e_{1,4}$ and $e_{3,5}$ cross. If $e_{1,5}$ is left, then there could be no more than two other left edges, and up to reflection, we can have either $e_{1,3}, e_{1,4}$, or $e_{1,3}, e_{3,5}$, or $e_{1,4}, e_{2,4}$. It is then easy to see that in each of the three cases, we get unwanted crossings one way or another after adding the remaining three edges. If apart from $e_{1,5}$, there is only one or no left edge, a contradiction is also readily found.
\end{proof}

\begin{remark} \label{r:nonsharp}
In contrast to the above lemma, when $n=4$, a tolerable drawing exists even when we take for $S$ the set of all subsets of $\{1,2,3,4\}$ as shown in Figure~\ref{F:n=4}.
\end{remark}

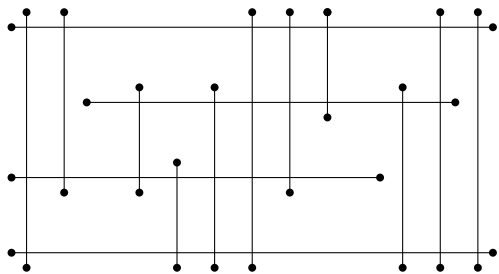
\begin{figure}[h!]
\begin{center}
\begin{tikzpicture}
\draw (-3.2,0) -- (3.2,0); \draw (-3.2,3) -- (3.2,3);
\draw (-2.2,2) -- (2.7,2); \draw (-3.2,1) -- (1.7,1);
\draw (0,-0.2) -- (0,3.2);
\draw (0.5,0.8) -- (0.5,3.2); \draw (1,1.8) -- (1,3.2);
\draw (-0.5,-0.2) -- (-0.5,2.2); \draw (-1,-0.2) -- (-1,1.2); \draw (-1.5,0.8) -- (-1.5,2.2);
\draw (-2.5,0.8) -- (-2.5,3.2); \draw (-3,-0.2) -- (-3,3.2);
\draw (2,-0.2) -- (2,2.2); \draw (2.5,-0.2) -- (2.5,3.2); \draw (3,-0.2) -- (3,3.2);
\foreach \y in {3.2} \foreach \x in {-3,-2.5,0,.5,,1,2.5,3}
\fill  ({\x},{\y}) circle (1.5pt);
\foreach \y in {0,3} \foreach \x in {-3.2,3.2}
\fill  ({\x},{\y}) circle (1.5pt);
\foreach \y in {2.2} \foreach \x in {-1.5,-.5,2}
\fill  ({\x},{\y}) circle (1.5pt);
\foreach \y in {2} \foreach \x in {-2.2,2.7}
\fill  ({\x},{\y}) circle (1.5pt);
\foreach \y in {1.8} \foreach \x in {1}
\fill  ({\x},{\y}) circle (1.5pt);
\foreach \y in {1.2} \foreach \x in {-1}
\fill  ({\x},{\y}) circle (1.5pt);
\foreach \y in {1} \foreach \x in {-3.2,1.7}
\fill  ({\x},{\y}) circle (1.5pt);
\foreach \y in {.8} \foreach \x in {-2.5,-1.5,.5}
\fill  ({\x},{\y}) circle (1.5pt);
\foreach \y in {-.2} \foreach \x in {-3,-1,-.5,0,2,2.5,3}
\fill  ({\x},{\y}) circle (1.5pt);
\end{tikzpicture}
\end{center}
\caption{Tolerable drawing for $n=4$ and $S$ the power set of $\{1,2,3,4\}$, with the subsets of cardinality $0$ and $1$ omitted}
\label{F:n=4}
\end{figure}

\begin{remark} \label{r:intolerableK5}
We have seen above that for $G$ equal to $K_4$, $K_{2,3}$ or $K_{1,1,3}$, every 2-realisable subset $A\subset PG$ can be 2-realised by a tolerable drawing. Consider the drawing of  $K_5$ in Figure \ref{F:k53}. It is a 2-realisation of the set $A=\{(1,2)(3,4), (1,3)(2,4), (1,4)(2,3)\}$, and it fails to be tolerable only because edge  (13) crosses edge (45) twice (in opposite directions); the other edges adjacent to vertex 5 do not cross independent edges of the $K_4$ having vertices $1,2,3,4$, and the drawing of this $K_4$ is tolerable. We suspect that there is no tolerable drawing of $K_5$ that 2-realises the set $A$, but we have not been able to prove this. Note that the problem is not even obviously a finite problem since there is no a priori bound on the number of crossings of adjacent edges. 

\begin{figure}[h!]
\begin{center}
\definecolor{blue}{rgb}{0,0,1}
\begin{tikzpicture}[scale=1.5];
\draw [thick,-]  (1.5,0) -- (0,1)  ;
\draw [thick,-]  (1.5,0) -- (0,-1)  ;
\draw [thick,-]  (0,1) -- (0,-1)  ;
\draw [thick,blue,-]  (-2,0) -- (-1,0)  ;
\draw [thick] plot [smooth, tension=.5] coordinates { (-1,0) (0,2)  (.75,.3)   (0,-1)};
\draw [thick] plot [smooth, tension=.5] coordinates { (-1,0) (0,-2)  (.75,-.3)   (0,1)};
\draw [thick] plot [smooth, tension=.5] coordinates { (-1,0) (0,-1.5)  (.4,-.4) (0,0) (-.4,.4) (0,1.5)    (1.5,0)};
\draw [thick,blue] plot [smooth, tension=.5] coordinates { (-2,0) (-1,.5) (-.5,0)  (0,-1)};
\draw [thick,blue] plot [smooth, tension=.5] coordinates { (-2,0) (-1,-.5) (-.5,0)  (0,1)};
\draw [thick,blue] plot [smooth, tension=.5] coordinates { (-2,0) (0,-2.5)   (1.5,0)};
\fill  (1.5,0) circle (2.5pt);
\fill  (0,-1) circle (2.5pt);
\fill  (0,1) circle (2.5pt);
\fill  (-1,0) circle (2.5pt);
\fill [blue]  (-2,0) circle(2.5pt);
\draw (1.7,0) node {$1$};
\draw (-2.2,0) node {$5$};
\draw (-1.13,0.23) node {$3$};
\draw (-.2,1) node {$4$};
\draw (-.2,-1) node {$2$};

\end{tikzpicture}

\end{center}
\caption{A bad drawing of $K_5$}
\label{F:k53}
\end{figure}
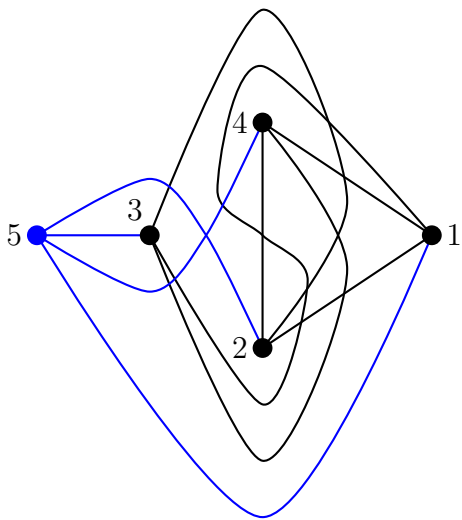

\end{remark}

\begin{remark} Another open question is as follows: does there exist a 
graph $G$ and an integer $i$ for which there exist  2-realisable subsets of $PG$ of cardinality $i$ but for none of these is there a tolerable drawing? One might say that such an integer $i$ is \emph{intolerable} for $G$.
The results of this paper show that no such phenomenon exists for any of $K_5,K_{3,3}$ or $ K_6$. Furthermore, there can be no intolerable integers for graphs that are disjoint union of edges. To see this, consider all our edges drawn as straight line segments passing through a single point, then replace each of them by a nearby parallel segment, so that all the crossing points are pairwise distinct, and then remove the crossings one by one by shortening the segments until we get the desired number of crossings.
\end{remark}


\section{Tolerable drawings for the graphs $K_5,K_{3,3}$ and $K_6$}\label{K5K33K6}

The complete graph $K_5$ has 10 edges and $PK_5$ has 15 elements. (We note in passing that for the complete graph $K_n$, the set $PK_n$ is the edge set of the Kneser graph $KG_{n,2}$).
Up to relabelling, only 5 of these subsets of $PK_5$ can be 2-realised by good drawings; see  \cite[Figure 3.1]{Raf} or \cite[Figure 1.7]{Sch2}. They each have 1, 3 or 5  crossings.  The first 3 drawings of Figure \ref{figure:k5} are good, and have 1, 3 and 5  crossings respectively.  The remaining drawings of Figure \ref{figure:k5} are tolerable and have 7, 9, 11, 13 and 15 independent crossings respectively. In particular, there is a tolerable drawing in which all 15 pairs of independent edges cross. To set this observation in context, note that $K_5$ cannot be drawn as a \emph{thrackle} in the plane \cite{CMN}; that is, it can't be drawn in the plane so that each pair of independent edges cross exactly once, and adjacent edges do not cross. Moreover, $K_5$ cannot be drawn as a \emph{generalized thrackle} in the plane \cite{CN09}; that is, it can't be drawn in the plane so that each pair of independent edges cross an odd number of times, and adjacent edges cross an even number of times. Also, $K_5$ cannot be drawn as a \emph{superthrackle} in the plane \cite{AS}; that is, it can't be drawn in the plane so that each pair of  edges (independent or not) cross exactly once.

As far as the 2-realisable subsets of $PK_5$ are concerned, note that $PK_5$ has $2^{15}$ subsets, and that the group acting here is the symmetry group $S_5$ of $K_{5}$. 
There is no difficulty in conducting the kind of symmetry reduction we employed for $K_{2,3}$ is Section~\ref{Section:K23}. For example,  one finds that up to relabelling, there are 9 essentially distinct subsets of  $PK_{5}$ of cardinality 3. They are:
\begin{enumerate}
\item[1.] $\{(12)(34),\ (13)(24),\ (14)(23)\}$
\item[2.] $\{(12)(34),\ (13)(24),\ (12)(35)\}$
\item[3.] $\{(12)(35),\ (13)(24),\ (14)(23)\}$
\item[4.] $\{(12)(34),\ (12)(35),\ (12)(45)\}$
\item[5.] $\{(12)(34),\ (12)(35),\ (15)(24)\}$
\item[6.] $\{(12)(34),\ (13)(25),\ (14)(25)\}$
\item[7.] $\{(12)(34),\ (15)(23),\ (14)(25)\}$
\item[8.] $\{(14)(23),\ (15)(23),\ (14)(25)\}$
\item[9.] $\{(14)(23),\ (13)(25),\ (15)(24)\}$.
\end{enumerate}
By Theorem~\ref{K5K33con}, each of these subsets is 2-realisable. However, the difficulty is in determining whether or not a given set can be 2-realised by a tolerable drawing. For $K_5$, we have not been able to resolve this problem, even for subsets of  $PK_{5}$ of cardinality 3; see Remark~\ref{r:intolerableK5} above, in which the subset $A$ in Figure \ref{F:k53} corresponds to case 1 above.

The graph $K_{3,3}$ has 9 edges and 18 pairs of independent edges. 
Harborth \cite{Harbo} determined that there are 102  good drawings of $K_{3,3}$ up to isomorphism; there are 1, 9, 33, 48, and 11 good drawings with 1, 3, 5, 7, and 9 crossings, respectively.
In Figure \ref{figure:k33}, the first 5 drawings are good. The remaining drawings are tolerable and have 11, 13, 15 and 17 independent crossings, respectively.

The graph $K_{6}$ has 15 edges and 45 pairs of independent edges. It is known that $K_{6}$ only has good drawings with $i$ crossings for  $3\le i\le 12$ and for $i=15$; see \cite{Raf,GH}. Figure \ref{F:k6;3-15} gives examples of such good drawings. 
Figures \ref{F:k6;16-24} through \ref{F:k6;37-40} give tolerable drawings having $i$  independent crossings for  $i=13,14 $ and $16\le i\le 40$. Note that in the drawings in Figures \ref{F:k6;25-36} and \ref{F:k6;37-40}, 
the idea is that one extends the red lines out and connects them up to a 6th vertex (at infinity, if one likes).  It is perhaps easiest to keep track of the independent crossings in these diagrams by comparing each drawing having $i$ independent crossings with the drawing having $i+3$ independent crossings.
Note that the first drawing in Figure \ref{F:k6;25-36} and the last drawing in Figure \ref{F:k6;37-40} have a 5-fold symmetry and are easy to understand; in each of these drawings the blue and black edges give a tolerable drawing of $K_5$ with 15 independent crossings. In the first case, each red line has 2 independent crossings with blue edges, while in the second case, each red line has 3 independent crossings with blue edges and 2 independent crossings with black edges.

\begin{figure}[h]
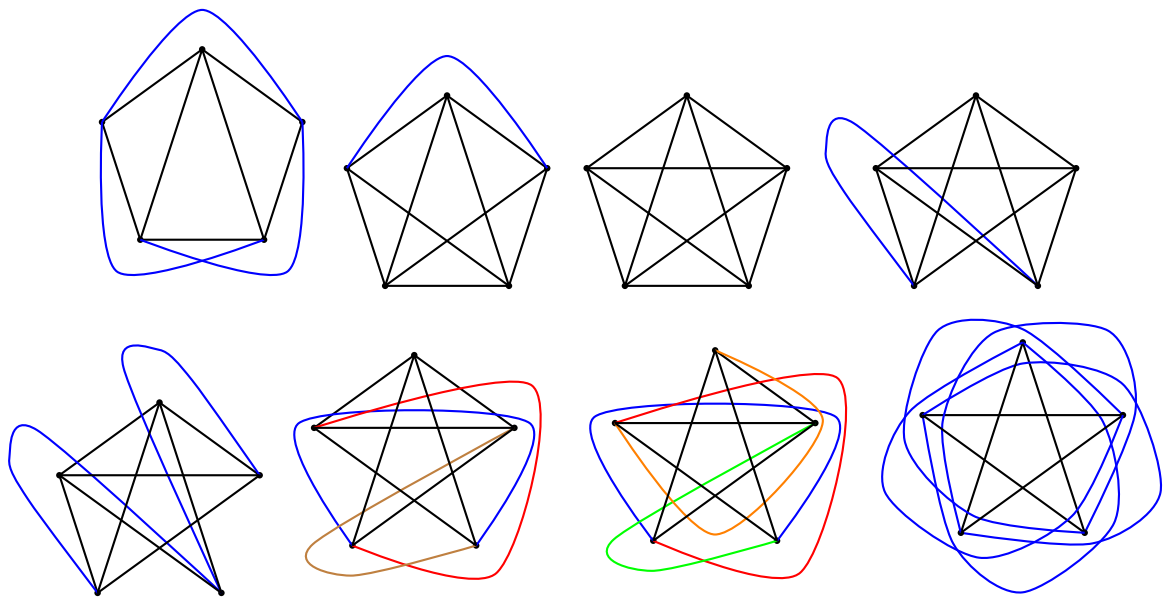

\centering

\caption{Tolerable drawings of $K_5$ }
\label{figure:k5}
\end{figure}

\begin{figure}[h]
\centering
%
\caption{Tolerable drawings of $K_{3,3}$ }
\label{figure:k33}
\end{figure}

\begin{figure}[h]
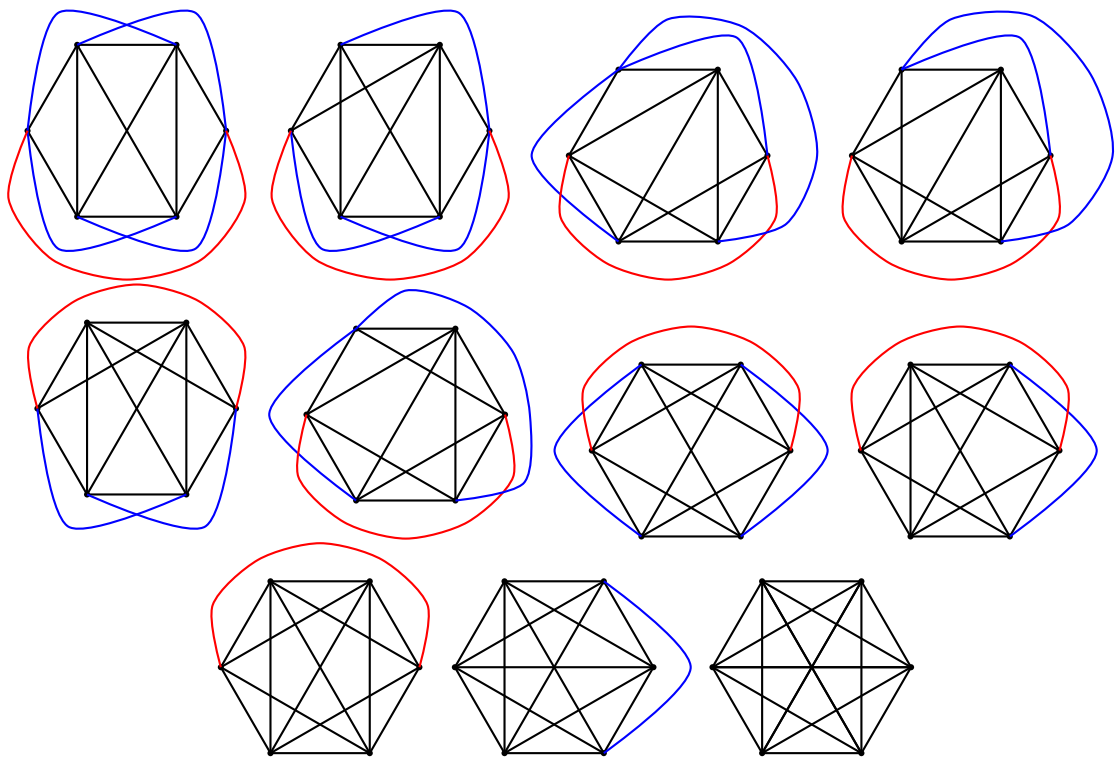

\centering
%
\caption{Good drawings of $K_6$ with 3 -- 12 and 15 crossings.}
\label{F:k6;3-15}
\end{figure}
\begin{figure}[h]
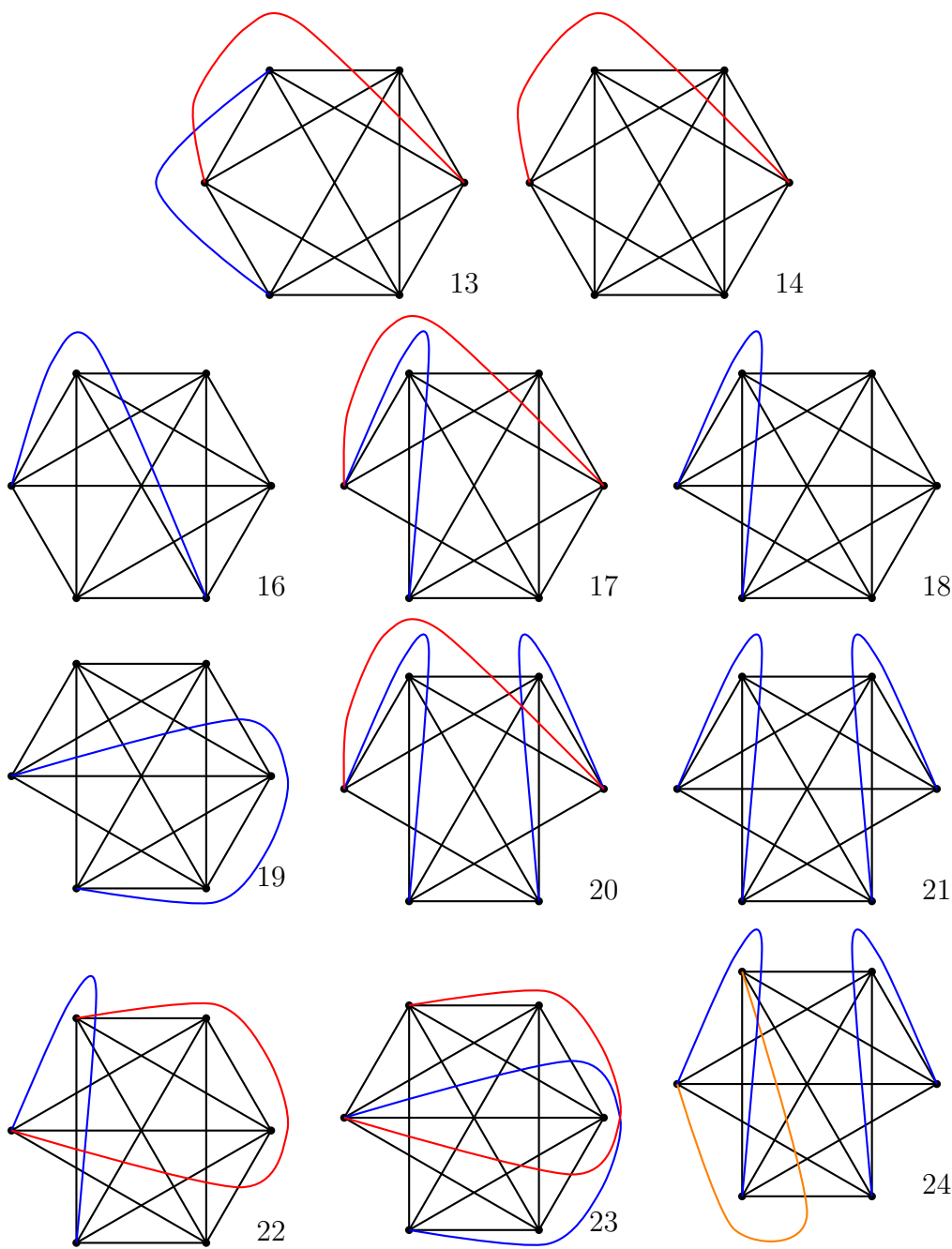

\centering
%
\caption{Tolerable drawings of $K_6$ with 13, 14 and 16 -- 24 independent crossings.}
\label{F:k6;16-24}
\end{figure}

\begin{figure}[h]
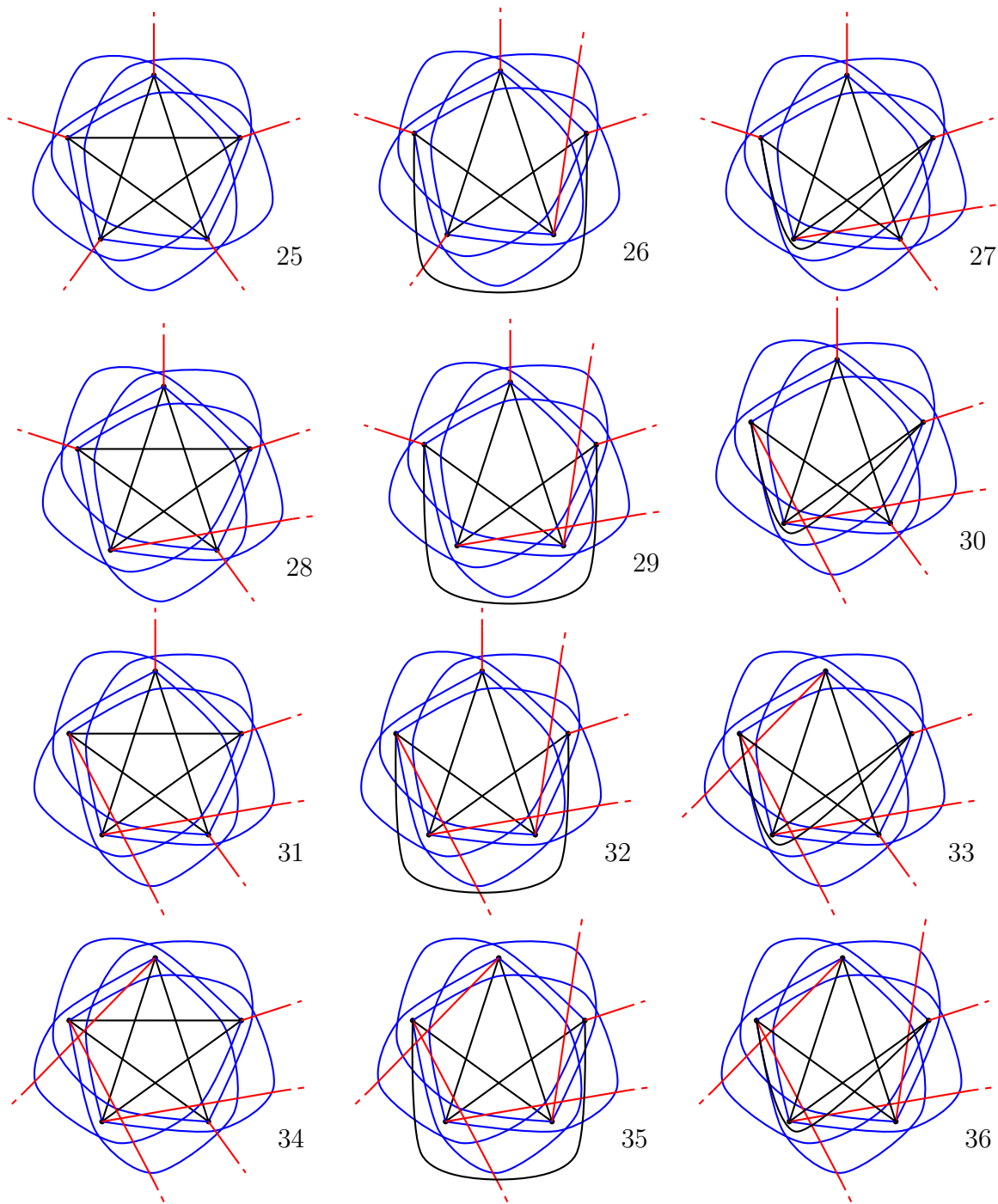

\centering
%
\caption{Tolerable drawings of $K_6$ with 25 to 36 independent crossings.}
\label{F:k6;25-36}
\end{figure}


 \begin{figure}[h]
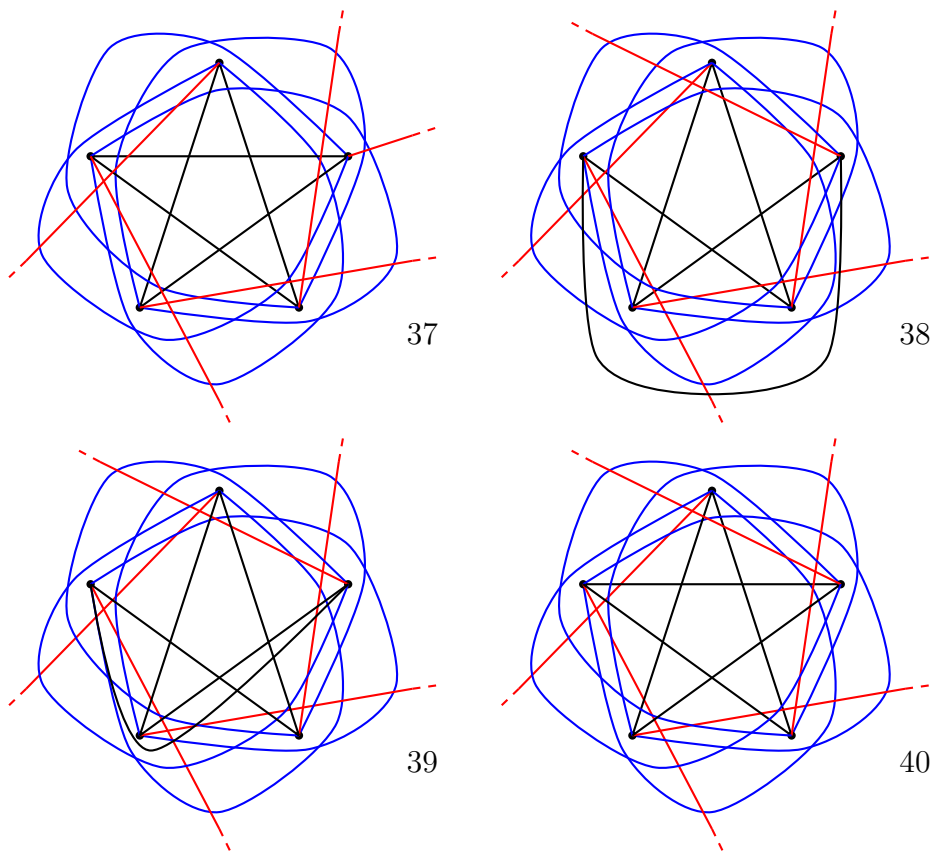

%
\caption{Tolerable drawings of $K_6$ with 37 to 40 independent crossings.}
\label{F:k6;37-40}
\end{figure}


\section{Deleted product cohomology and the van Kampen--Wu invariant}\label{Wu}
Consider a graph $G$, which we consider as a \emph{cell complex}; its ``cells'' are just its vertices and its edges. The \emph{deleted product space} $G^*$ of $G$ is the subcomplex of the cell complex $G\times G$  obtained by deleting all cells having nontrivial intersection with the diagonal.  A 1-cell in $G^*$ is of the form $(v,e)$ or $(e,v)$, where $v$ is a vertex of $G$ and $e$ is an edge that is not incident to $v$. For ease of presentation, we will denote these 1-cells $ve$ and $ev$ respectively. A 2-cell in $G^*$ is given by the pair $(e_1,e_2)$, where $e_1,e_2$ are independent edges. We will denote this 2-cell $e_1e_2$. Notice that the group $\Z_2$ acts  on $G^*$; the nontrivial involution is determined by the map $\tau$ on $G\times G$ sending $(x,y)$ to $(y,x)$. Since the map $\tau$ is fixed point free, the $\Z_2$ action is free. So the quotient $\overline{G^*}:=G^*/\Z_2$ is also a cell complex. For further information on the deleted product space of a graph, see Mark de Longueville's excellent text \cite{Long}, which provides a clear and clean exposition of this material.

We will be working with the cohomology of $\overline{G^*}$, or equivalently, with the cohomology of the $\Z_2$-invariant cocycles on $G^*$; we say that these cocycles  are \emph{symmetric}. Specifically, we work with the cohomology with coefficients in $\Z_2$. So a 2-cochain is given by a function from the set of 2-cells to $\Z_2$; that is, it is just a marking of the 2-cells with the symbols 0 or 1, or, if you like, it is determined by a subset of the set of 2-cells (given by the 2-cells labelled 1), or again as a formal sum over $\Z_2$ of 2-cells. Of course, a symmetric cochain is just a symmetric labelling. So, for example a symmetric 2-cochain is a sum of terms of the form $e_1e_2+e_2e_1$, where $e_1,e_2$ are independent edges. Similarly, a 1-cochain is a formal sum over $\Z_2$ of 1-cells, so a symmetric 1-cochain is a sum of terms of the form $ev+ve$.

The differential of a 1-cochain $ev$ (resp.~$ve$) is the sum of the 2-cochains of the form $ee'$ (resp.~$e'e$) where $e,e'$ are independent and $e'$ is incident to $v$. The differential of any 2-cochain is 0 (just because we are working with a cell complex of dimension 2), so every  2-cochain is a 2-cocycle. A 2-cocycle is \emph{exact} if it is the differential of a 1-cochain.

For a given drawing $f$ of $G$ in the plane, we define the symmetric 2-cocycle $\Phi_f(G)$ as follows: if $e_1,e_2$ are independent edges, we assign the number 1 to the 2-cells $e_1e_2$ and $e_2e_1$ if $e_1,e_2$ cross an odd number of times, and 0 otherwise. 

\begin{definition} The \emph{van Kampen  obstruction} is the symmetric cohomology class $\mathfrak o(G)$ of $\Phi_f(G)$.
\end{definition}

Where useful, we will also consider the corresponding form $\overline\Phi_f(G)$ on $\overline{G^*}$, and identify the symmetric cohomology class $\mathfrak o(G)$ with the element $[\overline\Phi_f(G)]\in H^2(\overline{G^*},\Z_2)$. 

\begin{KSWT}
The class $\mathfrak o(G)$ is well defined,  independent of the drawing. Moreover, if $\mathfrak o(G)=[\alpha]$ for some symmetric 2-cocycle $\alpha$, then there is a drawing $f$ of $G$ in the plane with $\alpha=\Phi_f(G)$.\end{KSWT}

Notice that in the obvious manner, every subset $A\in PG$ determines a symmetric 2-cocycle, and every symmetric 2-cocycle determines a subset $A\in PG$. The van Kampen--Shapiro--Wu Theorem can be reformulated in terms of 2-realisable subsets:

\begin{theorem}\label{C1}  Consider a graph $G$ with van Kampen symmetric cohomology class $\mathfrak o(G)\in H^2(\overline{G^*},\Z_2)$. Then a subset $A$ of $PG$ is a 2-realisable crossing set if and only if the element of the symmetric cohomology group $H^2(\overline{G^*},\Z_2)$ corresponding to $A$ is equal to $\mathfrak o(G)$.
\end{theorem}

A key application of this machinery is:

\begin{WT}
$\mathfrak o(G)=0$ if and only if $G$ is planar.
\end{WT}



\begin{remark} We should mention that it is well known that the early literature on the deleted product cohomology contained a number of errors. 
Errors in \cite{Pa} were observed by Ummel \cite{Umm} and Barnett and Farber \cite{BF}.
An error in  \cite{Co} was discussed by  Sarkaria  \cite{Sark} and Barnett and Farber \cite{BF}.  Sarkaria's paper  \cite{Sark} is a very attractive and readable work, but it also has errors that have been discussed by Skopenkov \cite{Skop}, van der Holst \cite{vdH2}, Barnett \cite{Ba},  and Schaefer \cite{Sch}.  
\end{remark}

\begin{remark} The history of this material might also merit a few comments, since there are confusions in the literature, possibly in part reflecting the divide between the fields of topology and combinatorics. The original notions came from van Kampen's 1932 paper \cite{vanK}, just 2 years after Kuratowski's famous 1930 Theorem 
 and 2 years prior to Hanani's version  \cite{Choj} of what is now known as the Hanani--Tutte Theorem.
The ``deleted product'' obstruction was  introduced by van Kampen  for measuring the non-embeddability of an $n$-dimensional simplicial complex in $\R^{2n}$, but only for $n \ge 3$. This work was later clarified, reformulated in cohomological language, and extended to dimension 2 in 1955 by Wu (subsequently translated into English \cite{WuI,WuII,WuIII} in 1958-1959, and later elaborated with a slightly different argument in his 1965 book \cite{Wu}, in English) and in 1957 by Shapiro \cite{Shap}. 
In fact, Wu was a well known topologist, in part because of his 1950 work on what is known as the \emph{Wu class}, and  Wu's 1955--1959 embedding work was widely read by topologists at the time; for example, see \cite{Haef1}. Later in 1970, as Levow \cite{Lev}  writes,``Tutte \cite{Tutte} rediscovered the van Kampen--Shapiro--Wu characterization of planar graphs''. 
 It seems that Tutte's paper brought this topic to the attention of combinatorialists, and motivated much of the subsequent investigations; although Tutte's paper uses topological arguments and some topological language (\emph{chain} and \emph{coboundary}), it never uses the word \emph{cohomology}, or even  \emph{differential}. The theorem known as the \emph{Hanani--Tutte Theorem}, which says that a graph is planar if it can be drawn in the plane so that each pair of independent edges cross an even number of times, is an immediate consequence of the result we have called the Wu--Tutte Theorem, and many would regard the two as being essentially the same result. Wu presented  the Wu--Tutte theorem in \cite{WuII} and also proved it in \cite[p.~210]{Wu};  he called it \emph{Kuratowski's Theorem}!   
\end{remark}



\section{Drawings of $K_5$ and $K_{3,3}$}\label{Ks}

We will require the following beautiful result, which was stated without proof in \cite{Ab}, and proved in Abrams' thesis \cite[Theorem 5.1]{Ath}.

\begin{theorem}\label{cs}
For $K_5$ and $K_{3,3}$, the deleted product is a closed surface. Moreover, $K_5$ and $K_{3,3}$ are the only graphs for which the deleted product is a closed surface.\end{theorem}

As the proof of this result is quite short, we include it for the convenience of the reader.

\begin{proof} In order for the deleted product $G^*$ of a graph $G$ to be a closed surface, one requires that each edge (i.e., 1-cell) in $G^*$ be incident with exactly two faces. This occurs precisely when, for each edge $e$ in $G$ and each vertex $v\in G$ that is not incident with $e$, there are exactly two edges in $G$ that are incident with $v$ and are independent of $e$. Rephrasing this yet again we obtain the following necessary and sufficient condition: 
\begin{enumerate}
\item[($*$)] for each edge $e$ in $G$, the graph $G- e$ obtained by deleting  $e$ and the interior of all adjacent edges, is a union of disjoint cycles.
\end{enumerate}
Now $K_5$ and $K_{3,3}$  satisfy condition ($*$), and so their deleted products are closed surfaces. On the other hand, if a graph $G$ satisfies ($*$), it must have at least 5 vertices. It is easy to verify that if $G$ has 5 vertices and satisfies ($*$), then it is $K_5$, and  if $G$ has 6 vertices and satisfies ($*$), then it is $K_{3,3}$. One checks readily that no graph with more than 6 vertices can satisfy ($*$).
\end{proof}

Kleitman proved that for odd $m,n$, any two drawings of $K_{m,n}$ (or of $K_{n}$) have equal numbers of independent crossings, mod 2 \cite{Kl2}.  
The result was independently proved for good drawings by Harborth \cite{Harbo}, and another proof, again for good drawings, was given by McQuillan and Richter  \cite{MR}.
In particular, all drawings of $K_5$ and $K_{3,3}$ have an odd number of independent crossings. 
The following result, which is a rewording of Theorem \ref{K5K33con}, is a converse to Kleitman's Theorem:

\begin{theorem}\label{T:sub}  For $G$ equal to $K_5$ or $K_{3,3}$, every odd subset $A$ of $PG$ is 2-realisable.
\end{theorem}

\begin{proof}
By Theorem~\ref{cs}, $\overline{G^*}$ is a closed surface. 
Consequently, the cohomology space $H^2(\overline{G^*},\Z_2)$ has dimension one.
Hence the exact 2-cocycles form a codimension one vector subspace of the space $Z$ of 2-cocycles in $\overline{G^*}$. So exactly half the 2-cocycles are exact, and half are not exact. Because $\overline{G^*}$ is a closed surface,  the differential of each 1-cell is the sum of two faces. It follows that all exact 2-cocycles are the sum of an even number of faces. Hence, because of their number, the set of exact 2-cocycles is precisely the set of 2-cocycles that are the sum of an even number of faces. Now consider the first drawings $f$ of $G$ in Figures~\ref{figure:k5} and \ref{figure:k33} respectively. It has only 1 independent crossing. So the corresponding cocycle $\overline\Phi_f(G)$ in $\overline{G^*}$ is a single face. It follows that $\mathfrak o(G)\not=0$. Thus for any drawing $g$ of $G$ in the plane, the corresponding cocycle $\overline\Phi_g(G)$  in $\overline{G^*}$ is the sum of an odd number of faces; i.e.,  the number of independent crossings is odd.  Conversely, by the van Kampen--Shapiro--Wu Theorem (see Theorem~\ref{C1}), every such cocycle is obtained from such a drawing.
\end{proof}

\section{Drawings of $K_6$}\label{higherKs}

According to our calculations, $\overline{K_6^*}$ has 45 faces and 60 edges, and the differential from the space of symmetric 1-cochains has rank 35. Using Theorem \ref{C1}, a 2-realisable crossing set of $K_6$ is given by a cochain of the form $\Phi_f(K_6)+\alpha$, for some drawing $f$ of $K_6$, where $\alpha$ is an element of the image of the differential. Choose $f$ to be the first drawing of Figure~\ref{F:k6;3-15}, and consider the $2^{35}$ possible elements $\alpha$. For each cochain $\Phi_f(K_6)+\alpha$, one just adds the numbers of 1's to obtain the ``cardinality''.  Performing this on a personal computer using Mathematica, we found that there is no 2-realisable crossing set for $K_6$ with cardinality in $\{0,1,2,41,42,43,44,45\}$. This establishes Theorem~\ref{K6con}.


We conclude this study with some further comments on the 2-realisable crossing set of $K_6$. First notice that there are tolerable drawings of $K_6$ with just 3 independent crossings, but not all subsets $A\subseteq P$ with 3 elements are 2-realisable. Indeed, in any drawing, every pair of independent 3-cycles must cross each other an even number of times, by the Jordan curve theorem. 
Since $\overline{K_6^*}$ has 45 faces and the differential has rank 35, the symmetric cohomology $H^2(\overline{K_6^*},\Z_2)$ has dimension 10. Notice also that there are 10 ways of separating the vertices of $K_6$ into two 3-element subsets. For each such partition, the two 3-element subsets give a copy of the disjoint union $C_3\sqcup C_3$ of two 3-cycles;  there are 9 potential crossing of the edges of one 3-cycle with the edges of the other, but as we just remarked, there must be an \emph{even} number of such crossings. This gives 10 conditions on $H^2(\overline{K_6^*},\Z_2)$ for $\mathfrak o(K_6)$. These conditions do not uniquely determine $\mathfrak o(K_6)$; for instance, the zero element satisfies them all, but  $\mathfrak o(K_6)\not=0$ as $K_6$ is not planar.
Consider the induced  $K_{3,3}$ subgraphs in $K_6$. For a 2-realisable set $A\subset PK_6$, the intersection $A\cap PK_{3,3}$ must have an odd number of elements, by Kleitman's theorem. There are 10 such induced $K_{3,3}$ subgraphs, so there are 10 conditions of this kind. However, it turns out that these conditions are not independent, so they also do not by themselves uniquely determine $\mathfrak o(K_6)$.  
Similarly, for the induced  $K_5$ subgraphs  in $K_6$,  for a 2-realisable set $A\subset PK_6$, the intersection $A\cap PK_5$ must have an odd number of elements. This gives 6 conditions on $H^2(\overline{K_6^*},\Z_2)$ for $\mathfrak o(K_6)$. Together, the above conditions are sufficient. One has:

\begin{theorem}\label{equivs}
 The subset $A$ of $PK_6$ is 2-realisable if and only if the following three conditions are satisfied:
\begin{enumerate}
\item for each induced $K_5$ subgraph of $K_6$, the intersection $A\cap PK_{5}$ has an odd number of elements,
\item for each induced $K_{3,3}$ subgraph of $K_6$, the intersection $A\cap PK_{3,3}$ has an odd number of elements,
\item for each induced $C_3\sqcup C_3$ subgraph of $K_6$, the intersection $A\cap P(C_3\sqcup C_3)$ has an even number of elements.
\end{enumerate}
\end{theorem}

The forward direction of this result is clear from the above discussion, and remains true if we replace $K_6$ by an arbitrary graph (in which case, one may replace in condition (c) the disjoint union of two 3-cycles by a disjoint union of any two cycles). We established the sufficiency of the conditions directly by computer computation. The result is closely related to van der Holst's theorem \cite[Theorem~4]{vdH}, which holds for arbitrary graphs $G$ and which says that the symmetric deleted product cohomology $H^2(\overline{G^*},\Z_2)$ is generated 
by subdivisions of  $K_5$'s and $K_{3,3}$'s, and by  2-tori resulting from pairs of disjoint cycles. Note that $K_6$ has subgraphs which are nontrivial subdivisions of $K_5$ but it was not necessary to consider these in Theorem~\ref{equivs}.

As we will now explain,  it is also possible to establish Theorem~\ref{K6con} without reliance on computer computations, or the use of the deleted product cohomology, by using the forward direction of Theorem~\ref{equivs} (which didn't require computer computation). Suppose that  $A$ is a 2-realisable crossing set  for $K_6$. One needs to show that $3\le \# A\le 40$. If  $A$ has only 0, 1 or 2 elements then it is easy to find an induced $K_{3,3}$ subgraph for which $A\cap PK_{3,3}$ is empty, contradicting condition (b) of Theorem~\ref{equivs}. So it remains to show that $\# A \le 40$. Indeed, one had the following general result.

\begin{lemma}
If $n\geq 6$ and $A$ is a 2-realisable crossing set  for $K_n$, then $\# A \le \lfloor \frac83 \binom{n}{4}\rfloor$.\end{lemma}

\begin{proof}
By condition (c) of Theorem~\ref{equivs}, for a plane drawing of a complete graph $K_n$, the number of crossings of any two \emph{independent 3-cycles} (i.e., 3-cycles which have no common vertices) must be even. Therefore if a subset $A \subset PK_n$ is $2$-realisable, then from the nine pairs of independent edges defined by a pair of independent 3-cycles, at least one pair (and in general, an odd number of pairs) does not belong to $A$. For $n \ge 6$ one has $\frac12 \binom{n}{3} \binom{n-3}{3}$ pairs of independent 3-cycles. A pair of independent edges belongs to $(n-4)(n-5)$ pairs of independent 3-cycles, and so the complement $A'$ to $A$ in $PK_n$ must be of cardinality at least 
\[
\frac{\frac12 \binom{n}{3} \binom{n-3}{3}}{(n-4)(n-5)}=\frac13 \binom{n}{4}.
\]
Thus, since $\# PK_n = 3 \binom{n}{4}$, we have  $\# A \le  3 \binom{n}{4} -\frac13 \binom{n}{4} =\frac83 \binom{n}{4}$.
\end{proof}
In the case $n=6$, the above lemma gives $\# A \le 40$, as required.





\bibliographystyle{amsplain}
\bibliography{baddraw}

\newcommand{\noopsort}[1]{} \newcommand{\printfirst}[2]{#1}
  \newcommand{\singleletter}[1]{#1} \newcommand{\switchargs}[2]{#2#1}
\providecommand{\bysame}{\leavevmode\hbox to3em{\hrulefill}\thinspace}
\providecommand{\MR}{\relax\ifhmode\unskip\space\fi MR }
\providecommand{\MRhref}[2]{%
  \href{http://www.ams.org/mathscinet-getitem?mr=#1}{#2}
}
\providecommand{\href}[2]{#2}
\begin{thebibliography}{10}

\bibitem{Ath}
Aaron Abrams, \emph{Configuration spaces and braid groups of graphs}, Ph.D.
  thesis, University of California Berkley, 2000.

\bibitem{Ab}
\bysame, \emph{Configuration spaces of colored graphs}, Geom. Dedicata
  \textbf{92} (2002), 185--194.

\bibitem{Aetc}
B.M. \'Abrego, O.~Aichholzer, S.~Fern\'andez-Merchant, T.~Hackl, J.~Pammer,
  A.~Pilz, P.~Ramos, G.~Salazar, and B.~Vogtenhuber, \emph{All good drawings of
  small complete graphs}, Proc. 31st European Workshop on Computational
  Geometry (EuroCG 2015), Ljubljana, Slovenia, 2015, pp.~57--60.

\bibitem{AS}
Dan Archdeacon and Kirsten Stor, \emph{Superthrackles}, Australas. J. Combin.
  \textbf{67} (2017), 145--158.

\bibitem{Ba}
Kathryn Barnett, \emph{The configuration space of two particles moving on a
  graph}, Ph.D. thesis, Durham University, 2010.

\bibitem{BF}
Kathryn Barnett and Michael Farber, \emph{Topology of configuration space of
  two particles on a graph. {I}}, Algebr. Geom. Topol. \textbf{9} (2009),
  no.~1, 593--624.

\bibitem{CMN}
Grant Cairns, Margaret McIntyre, and Yury Nikolayevsky, \emph{The {T}hrackle
  conjecture for {$K\sb 5$} and {$K\sb {3,3}$}}, Towards a theory of geometric
  graphs, Contemp. Math., vol. 342, Amer. Math. Soc., Providence, RI, 2004,
  pp.~35--54.

\bibitem{CN09}
Grant Cairns and Yury Nikolayevsky, \emph{Generalized thrackle drawings of
  non-bipartite graphs}, Discrete Comput. Geom. \textbf{41} (2009), no.~1,
  119--134.

\bibitem{Choj}
Chaim Chojnacki (Haim~Hanani), \emph{{\"U}ber wesentlich unpl\"attbare kurven
  im drei-dimensionalen raume}, Fundamenta Mathematicae \textbf{23} (1934),
  135--142.

\bibitem{Co}
Arthur~H. Copeland, Jr., \emph{Homology of deleted products in dimension one},
  Proc. Amer. Math. Soc. \textbf{16} (1965), 1005--1007.

\bibitem{Long}
Mark de~Longueville, \emph{A course in topological combinatorics},
  Universitext, Springer, New York, 2013.

\bibitem{GH}
Hans-Dietrich O.~F. Gronau and Heiko Harborth, \emph{Numbers of nonisomorphic
  drawings for small graphs}, Proceedings of the {T}wentieth {S}outheastern
  {C}onference on {C}ombinatorics, {G}raph {T}heory, and {C}omputing ({B}oca
  {R}aton, {FL}, 1989), vol.~71, 1990, pp.~105--114.

\bibitem{guy}
Richard~K. Guy, \emph{Crossing numbers of graphs}, Graph theory and
  applications ({P}roc. {C}onf., {W}estern {M}ichigan {U}niv., {K}alamazoo,
  {M}ich., 1972; dedicated to the memory of {J}. {W}. {T}. {Y}oungs), Lecture
  Notes in Math., Vol. 303, Springer, Berlin, 1972, pp.~111--124.

\bibitem{Haef1}
Andr\'e Haefliger, \emph{Points multiples d'une application et produit cyclique
  r\'eduit}, Amer. J. Math. \textbf{83} (1961), 57--70.

\bibitem{Harbo}
Heiko Harborth, \emph{Parity of numbers of crossings for complete {$n$}-partite
  graphs}, Math. Slovaca \textbf{26} (1976), no.~2, 77--95.

\bibitem{Kl2}
D.~J. Kleitman, \emph{A note on the parity of the number of crossings of a
  graph}, J. Combinatorial Theory Ser. B \textbf{21} (1976), no.~1, 88--89.

\bibitem{Lev}
Roy~B. Levow, \emph{On {T}utte's algebraic approach to the theory of crossing
  numbers}, pp.~315--324, Florida Atlantic Univ., Boca Raton, Fla., 1972.

\bibitem{MR}
Dan McQuillan and R.~Bruce Richter, \emph{A parity theorem for drawings of
  complete and complete bipartite graphs}, Amer. Math. Monthly \textbf{117}
  (2010), no.~3, 267--273.

\bibitem{Neu}
Peter~M. Neumann, \emph{A lemma that is not {B}urnside's}, Math. Sci.
  \textbf{4} (1979), no.~2, 133--141.

\bibitem{Pa}
C.~W. Patty, \emph{The fundamental group of certain deleted product spaces},
  Trans. Amer. Math. Soc. \textbf{105} (1962), 314--321.

\bibitem{PSS2008}
Michael~J. Pelsmajer, Marcus Schaefer, and Daniel \v{S}tefankovi\v{c},
  \emph{Odd crossing number and crossing number are not the same}, Discrete
  Comput. Geom. \textbf{39} (2008), no.~1-3, 442--454.

\bibitem{Raf}
Nabil~H. Rafla, \emph{The good drawings {${D}_n$} of the complete graph
  {${K}_n$}}, Ph.D. thesis, McGill University, Montreal, 1988.

\bibitem{Sark}
K.~S. Sarkaria, \emph{A one-dimensional {W}hitney trick and {K}uratowski's
  graph planarity criterion}, Israel J. Math. \textbf{73} (1991), no.~1,
  79--89.

\bibitem{Sch}
Marcus Schaefer, \emph{Hanani-{T}utte and related results},
  Geometry---intuitive, discrete, and convex, Bolyai Soc. Math. Stud., vol.~24,
  J\'anos Bolyai Math. Soc., Budapest, 2013, pp.~259--299.

\bibitem{Sch2}
\bysame, \emph{Crossing numbers of graphs}, Discrete Mathematics and its
  Applications (Boca Raton), CRC Press, Boca Raton, FL, 2018.

\bibitem{Shap}
Arnold Shapiro, \emph{Obstructions to the imbedding of a complex in a euclidean
  space. {I}. {T}he first obstruction}, Ann. of Math. (2) \textbf{66} (1957),
  256--269.

\bibitem{Skop}
Mikhail Skopenkov, \emph{On approximability by embeddings of cycles in the
  plane}, Topology Appl. \textbf{134} (2003), no.~1, 1--22.

\bibitem{Tutte}
W.~T. Tutte, \emph{Toward a theory of crossing numbers}, J. Combinatorial
  Theory \textbf{8} (1970), 45--53.

\bibitem{Umm}
Brian~R. Ummel, \emph{Some examples relating the deleted product to
  imbeddability}, Proc. Amer. Math. Soc. \textbf{31} (1972), 307--311.

\bibitem{vdH}
Hein van~der Holst, \emph{Algebraic characterizations of outerplanar and planar
  graphs}, European J. Combin. \textbf{28} (2007), no.~8, 2156--2166.

\bibitem{vdH2}
\bysame, \emph{A polynomial-time algorithm to find a linkless embedding of a
  graph}, J. Combin. Theory Ser. B \textbf{99} (2009), no.~2, 512--530.

\bibitem{vanK}
E.~R. van Kampen, \emph{Komplexe in euklidischen {R}\"aumen}, Abh. Math. Sem.
  Univ. Hamburg \textbf{9} (1933), no.~1, 72--78 and 52--53.

\bibitem{WuI}
Wu~Wen-Ts\"un, \emph{On the realization of complexes in {E}uclidean spaces.
  {I}}, Sci. Sinica \textbf{7} (1958), 251--297.

\bibitem{WuII}
\bysame, \emph{On the realization of complexes in {E}uclidean spaces. {II}},
  Sci. Sinica \textbf{7} (1958), 365--387.

\bibitem{WuIII}
\bysame, \emph{On the realization of complexes in {E}uclidean spaces. {III}},
  Sci. Sinica \textbf{8} (1959), 133--150.

\bibitem{Wu}
\bysame, \emph{A theory of imbedding, immersion, and isotopy of polytopes in a
  {E}uclidean space}, Science Press, Peking, 1965.

\end{thebibliography}

\end{document}